\documentclass[11pt]{article}
\usepackage{amssymb, authblk}
\usepackage{amsmath}
\usepackage{fullpage, bbm}
\usepackage{amsthm, natbib, hyperref}
\usepackage{graphicx, capt-of}
\usepackage{verbatim}
\usepackage{algorithm}% http://ctan.org/pkg/algorithms
\usepackage{algpseudocode}% http://ctan.org/pkg/algorithmicx
\usepackage[dvipsnames]{xcolor}

\usepackage{hyperref}[]
\hypersetup{
    colorlinks=true,
    linkcolor=blue,
    filecolor=magenta,      
    urlcolor=cyan,
      citecolor=blue,
    }

\usepackage{cleveref}

\newtheorem{thm}{Theorem}

\newtheorem{lemma}[thm]{Lemma}
\newtheorem{prop}[thm]{Proposition}

\newtheorem{assumption}{Assumption}

\DeclareMathOperator*{\argmin}{arg~min}

\allowdisplaybreaks

\algnewcommand\algorithmicinput{\textbf{INPUT:}}
\algnewcommand\INPUT{\item[\algorithmicinput]}
\algnewcommand\algorithmicoutput{\textbf{OUTPUT:}}
\algnewcommand\OUTPUT{\item[\algorithmicoutput]}

\date{\vspace{-5ex}}

\title{Spectral Analysis of High-dimensional Time Series}

\author[1]{Mark Fiecas}
\author[2]{Chenlei Leng}
\author[3]{Weidong Liu}
\author[4]{Yi Yu}

\affil[1]{\small School of Public Health, University of Minnesota}
\affil[2]{\small Department of Statistics, University of Warwick}
\affil[3]{\small Institute of Natural Sciences, Shanghai Jiao Tong University}
\affil[4]{\small School of Mathematics, University of Bristol}

\begin{document}
\maketitle

\begin{abstract}
A useful approach for analysing multiple time series is via characterising their spectral density matrix as the frequency domain analog of the covariance matrix.  When the dimension of the time series is large compared to their length, regularisation based methods can overcome the curse of dimensionality, but the existing ones lack theoretical justification.  This paper develops the first non-asymptotic result for characterising the difference between the sample and population versions of the spectral density matrix, allowing one to justify a range of high-dimensional models for analysing time series.  As a concrete example, we apply this result to establish the convergence of the smoothed periodogram estimators and sparse estimators of the inverse of spectral density matrices, namely precision matrices.  These results, novel in the frequency domain time series analysis, are corroborated by simulations and an analysis of the Google Flu Trends data.\\

%In this paper, we study the fixed sample results for frequency domain time series in high dimension.  The theoretical framework is based on the functional dependency \cite{Wu2005}, which covers a large class of common time series models.  We provide convergence results for the smoothed periodogram estimators and sparse estimators of the inverse of spectral density matrices, namely precision matrices.  The results are novel in the frequency domain time series analysis.  The techniques and results in this paper can be extended to a wide range of high-dimensional inference problems for frequency domain time series analysis.


\end{abstract}
%=================
\section{Introduction}\label{sec-intro}
Spectral density matrices play a large role in characterising the second order properties of multivariate time series. The spectral density matrix is the frequency domain analog of the covariance matrix, and describes the variance in each dimension or the covariance between dimensions that can be attributed to oscillations in the data within certain frequencies. Just as how partial correlations between the dimensions can be extracted as a function of the inverse of a covariance matrix, conditional relationships attributable to variations in the oscillations of the data can be obtained from the inverse of the spectral density matrix \citep{Dahlhaus00}. Thus, it is necessary to obtain a positive-definite estimate of the spectral density matrix, but this can be challenging whenever the dimensionality of the time series is relatively large compared to the length of the time series.

There have only been a few papers dedicated to developing rigorous theory in the context of a high-dimensional time series.  For instance, \cite{Davis16} and \cite{Guo16} both developed methods to give sparse estimates of the parameters of a vector autoregressive (VAR) model, and \cite{Basu15} studied the theoretical properties of regularised estimates of the parameters of a broad class of time series models. These recent works, however, focused primarily on time series models in the time domain, yet, there remains a critical gap in theoretical investigations on frequency domain methodologies. Nevertheless, many authors have been proposing frequency domain methodologies despite the lack of theoretical justifications.  For instance, \cite{Fiecas11}, \cite{Fiecas14}, and \cite{Schneider-Luftman16} developed variations of a shrinkage framework developed by \cite{Bohm:vonSachs:2009} for data-driven $\ell_2$-penalised estimation, and applied their ideas to neuroimaging data; motivated by gene regulatory networks as well as econometrics, \cite{Jung2015} developed a graphical lasso approach for estimating a graphical model for high-dimensional time series data in the spectral domain; \cite{BarigozziHallin2017} utilised a dynamic factor model to study the volatility of high dimensional financial series.  This stream of methodological % and applied 
papers have deep roots in the application areas, where we are aware of the demand of estimating high-dimensional spectral density matrices and its inverse.

The aim of this paper is to study the theoretical behaviours of estimators of the spectral density matrix and its inverse in high dimension. We summarise the main contributions of this paper.
% \begin{itemize}

First, it is arguable that the most important ingredient in high-dimensional statistical inference, in contrast with classical ones, is the fixed-sample results.  To be specific, in order to allow for high dimensions, a common practice is to exploit concentration inequalities, then to provide fixed-sample results to control the differences between the sample and the population versions, and finally to use union bound arguments to derive desirable results.  To the best of our knowledge, this paper is the first to show such fixed-sample results on the error control of the smoothed periodogram matrices in \Cref{thm-key}.  This is a challenging task, and the main difficulty in developing such methods comes from the fact that the text book results on frequency domain time series are limited to asymptotic results only \citep{Brillinger1981,BrockwellDavis2006}.

 Second, once the fixed-sample results are established, a wide range of high-dimensional statistics methods are ready to be justified, including estimation, prediction and inference tools.  In this paper, we use the sparse precision matrix estimation problem as an example, and demonstrate the theoretical (see \Cref{thm-main}) and numerical performances of applying the constrained $\ell_1$-minimisation for inverse matrix estimation \citep[\textsc{clime},][]{CaiEtal2011} to spectral analysis of time series data.  We would like to mention that the possible applications of \Cref{thm-key} are way beyond \Cref{thm-main}, while we use the sparse precision matrix estimation as an example.

The rest of this paper is organised as follows.  In \Cref{sec-method}, we explain the methodology used in this paper.  The theoretical results are collected in \Cref{sec-theory}, including two main theorems.  The technical details thereof can be found in the Appendix.  In \Cref{sec-numerical}, we demonstrate the numerical performances of our proposed methods, via simulations and real data analysis.  

% \end{itemize}

%=================
\section{Methodology}\label{sec-method}

\subsection{Framework and notation}\label{sec-framework}

%In this paper, we aim to provide the theoretical justification for the high-dimensional frequency domain time series analysis.  
In order to study the theoretical performances, we adopt the functional dependency framework  \citep{Wu2005}.  Let $\boldsymbol{X}_t=(X_{t,1},\ldots,X_{t,p})^{\top} \in \mathbb{R}^p$ be centred random vectors  satisfying
	\begin{eqnarray}\label{frame1}
	\boldsymbol{X}_t = G(\ldots, \boldsymbol{e}_{t-1}, \boldsymbol{e}_{t}) =: G(\mathcal{F}_{t}),
	\end{eqnarray}
	where $\boldsymbol{e}_{t}$ are i.i.d. random vectors, $\mathcal{F}_{t} = (\ldots, \boldsymbol{e}_{t-1}, \boldsymbol{e}_{t})$, and $G(\mathcal{F}_{t}) = (g_1(\mathcal{F}_{t}), \ldots, g_{p}(\mathcal{F}_{t}))^{\top}$.  With this notation, we have $X_{t,i}=g_{i}(\mathcal{F}_{t})$ for each $i \in \{1, \ldots, p\}$.   Let $\tilde{\boldsymbol{e}}_{0}, \{\boldsymbol{e}_{t}, t\in \mathbb{Z}\}$ be i.i.d. random vectors.  For $t = 1, 2, \ldots$, define $\tilde{\mathcal{F}}_{t}=(\ldots, \boldsymbol{e}_{-1},\tilde{\boldsymbol{e}}_{0}, \boldsymbol{e}_{1}, \ldots, \boldsymbol{e}_{t})$, i.e. replace $\boldsymbol{e}_{0}$ with $\tilde{\boldsymbol{e}}_{0}$ in $\mathcal{F}_{t}$.  Define $X'_{t,i} = g_{i}(\tilde{\mathcal{F}}_{t})$ and 
 	\begin{equation}\label{eq-theta-intro}
	\theta_{t, i} = \bigl(\mathbb{E} |X_{t,i} - X'_{t,i}|^2\bigr)^{1/2},
	\end{equation}
	which is used as a dependency measure.  It has been pointed out in \cite{Wu2005} that a large family of common time series models can be characterised by imposing proper conditions on \eqref{eq-theta-intro}.
	
%We further define the autocovariance matrix function $\Gamma(h) = \mathrm{Cov}(\boldsymbol{X}_t, \boldsymbol{X}_{t+h}) = \mathbb{E}(\boldsymbol{X}_t\boldsymbol{X}^{\top}_{t+h})$, $h \in \mathbb{Z}_+$.  Under mild conditions, $\{\boldsymbol{X}_t\}$ has a continuous spectral density matrix given by
%	\[
%	f(\omega) = \sum_{h\in\mathbb{Z}}\Gamma(h)\exp(-\imath 2\pi\omega h), \quad \omega \in [-1/2, 1/2].
%	\]
%	The aim of this paper is to study the theoretical behaviours of estimators of the spectral density matrix and its inverse in high dimension.  We refer to the inverse as a precision matrix.

For rest of this paper, for any vector $\boldsymbol{v} = (v_1, \ldots, v_m)^{\top} \in \mathbb{C}^m$, let $\|\boldsymbol{v}\|_q := \bigl(\sum_{i=1}^m |v_i|^q \bigr)^{1/q}$ be the $\ell_q$-norm of $\boldsymbol{v}$; for any matrix $A = (A_{ij})_{i,j=1}^p \in \mathbb{C}^{p\times p}$, let $\|A\|_{w} = \sup_{\boldsymbol{v}:\, \|\boldsymbol{v}\|_w \leq 1} \|A\boldsymbol{v}\|_w$.  We use the sparsity definition in \cite{CaiEtal2016} to characterise the sparsity of precision matrices, i.e. let the parameter space $\mathcal{G}_q(c_{n, p}, M_{n, p})$ be denoted by
	\begin{equation}\label{eq-sparsity-definition}
		\mathcal{G}_q(c_{n, p}, M_{n, p}) := \left\{\begin{array}{c}
			\Theta = (\Theta_{ij})_{i, j = 1}^p: \, \max_{j = 1, \ldots, p} \sum_{i=1}^p |\Theta_{ij}|^q \leq c_{n, p},\\
			\|\Theta\|_1 \leq M_{n, p}, \, \lambda_{\max}(\Theta)/\lambda_{\min}(\Theta) \leq M_1, 
			\end{array} \right\},	
	\end{equation}
	where $0 \leq q < 1$, $c_{n, p}$ and $M_{n, p}$ are potentially diverging as $n$ and $p$ grow.

\subsection{The Sparse Inverse Periodogram Estimator}
In the following sections, we define the spectral density matrix, introduce the estimators thereof, and propose a method to estimate the inverse of the high-dimensional spectral density matrix for any arbitrary frequency.  We convert the time domain time series data into frequency domain using the discrete Fourier transform, which results in the data being a complex-valued vector.  Motivated by the properties of the complex-valued normal distribution, we separate the real and imaginary parts of the transformed data and double the dimension of the vectors.  At each frequency point, we adopt a moving window and construct the estimator of the inverse of the periodogram, based on the \textsc{clime} estimator proposed in \cite{CaiEtal2011}.  The detailed algorithm is in Algorithm~\ref{alg-1}.  %In the rest of this section, 
We will first state our algorithm, and explain the details regarding the smoothed periodogram and its inverse in Sections~\ref{sec-21} and \ref{sec-22}, respectively.

\begin{algorithm}[htbp]
\begin{algorithmic}
\Procedure{SIPE}{$\{\boldsymbol{X}_i \in \mathbb{R}^p\}_{i=1}^n$, $h$}	
	\For{$j \in \{-\lfloor (n-1)/2 \rfloor, \ldots, \lfloor n/2\rfloor \}$}
		\State $\omega_j \leftarrow j/n$
		\State $\boldsymbol{d}(\omega_j) \leftarrow \sum_{t=1}^n\boldsymbol{X}_t\exp(- \imath 2\pi\omega_j t)$ \Comment{$\boldsymbol{d}(\omega_j) \in \mathbb{C}^p$}
	\EndFor
	\State $D^{\mathrm{C}} \leftarrow \big(\boldsymbol{d}(\omega_{-\lfloor (n-1)/2}), \ldots, \boldsymbol{d}(\omega_{\lfloor n/2\rfloor})\big)^{\top}$ \Comment{$D^{\mathrm{C}} \in \mathbb{C}^{n \times p}$}
	\State $D \leftarrow (\Re(D^{\mathrm{C}}), \Im(D^{\mathrm{C}}))$ \Comment{$D \in \mathbb{R}^{n \times 2p}$}
	\For{$j \in \{-\lfloor (n-1)/2 \rfloor, \ldots, \lfloor n/2\rfloor \}$}
		\State ind $\leftarrow (j - h, j -h+1, \ldots, j+ h) \mod (n+1)$
		\State $\left(\begin{array}{cc}
		 A_1 & B_1 \\
		 B_2 & A_2 
		 \end{array}\right) \leftarrow \textsc{clime}(D_{\mathrm{ind}})$ \Comment{$D_{\mathrm{ind}} \in \mathbb{R}^{|\mathrm{ind}| \times 2p}$, $A_1, A_2, B_1, B_2 \in \mathbb{R}^{p \times p}$}
		\State $\Theta_j \leftarrow (A_1 + A_2)/2 + \imath (B_1 - B_2)/2$ \Comment{$\Theta_j \in \mathbb{C}^{p \times p}$}
	\EndFor
	\State \textbf{return} $\{\Theta_i \in \mathbb{C}^{p \times p}\}_{i=1}^n$.
\EndProcedure
\caption{Sparse Inverse Periodogram Estimation. }\label{alg-1}	
\end{algorithmic}
\end{algorithm}

In \Cref{alg-1}, $\Re(\cdot)$ and $\Im(\cdot)$ denote the real and imaginary parts of an object, respectively, and preserve the same format of the object.  In our case, the input $D^{\mathrm{C}} \in \mathbb{R}^{n \times p}$, and therefore $\Re(D^{\mathrm{C}}), \Im(D^{\mathrm{C}}) \in \mathbb{R}^{n \times p}$.  As for the algorithm \textsc{clime}, see \Cref{sec-22} and \cite{CaiEtal2011} for details.

\subsection{Real-valued smoothed periodogram estimators}\label{sec-21}

Let $\{\boldsymbol{X}_t\}_{t \in \mathbb{Z}}$ be a $p$-variate mean zero  stationary real-valued time series with autocovariance matrix function $\Gamma(h) = \mathrm{Cov}(\boldsymbol{X}_t, \boldsymbol{X}_{t+h}) = \mathbb{E}(\boldsymbol{X}_t\boldsymbol{X}^{\top}_{t+h})$, for $h \in \mathbb{Z}$.  Under these conditions, $\{\boldsymbol{X}_t\}$ has a continuous spectral density matrix given by
	\[
		f(\omega) = \sum_{h\in\mathbb{Z}}\Gamma(h)\exp(-\imath 2\pi\omega h), \quad \omega \in [-1/2, 1/2].
	\]
	Given an interval of the whole time series, namely $\{\boldsymbol{X}_t\}_{t = 1, \ldots, n}$, the periodogram defined at the Fourier frequencies $\{\omega_j = j/n, \, -\lfloor(n-1)/2 \rfloor \leq j \leq \lfloor n/2\rfloor\}$ by $P_n(\omega_j) = n^{-1}\boldsymbol{d}(\omega_j)\boldsymbol{d}^*(\omega_j)$, where $\boldsymbol{d}(\omega_j) = \sum_{t=1}^n\boldsymbol{X}_t\exp(-\imath 2\pi\omega_j t)$, and for any complex-valued vector $\boldsymbol{v}$, $\boldsymbol{v}^*$ denotes $\overline{\boldsymbol{v}}^{\top}$, i.e. the conjugate transpose of $\boldsymbol{v}$.

When $p = 1$, it is known that $\mathbb{E}(P_n(\omega))$ converges uniformly to $f(\omega)$ on $[-1/2, 1/2]$ \citep[e.g.][Proposition~10.3.1]{BrockwellDavis2006}, but $P_n(\omega)$ does not converge in probability to $f(\omega)$ as $T \to \infty$ \citep[e.g.][Theorem~10.3.2]{BrockwellDavis2006}.  A common remedy is to use the smoothed periodogram, given by
	\[
		\widetilde{f}_{n}(\omega_j) = \frac{1}{2M_n+1}\sum_{|k|\leq M_n}P_n(\omega_{j+k}).
	\]
	When $p=1$, it can be shown that if $M_n \to \infty$ and $M_n/n \to 0$ as $n \to \infty$, $\widetilde{f}_n(\omega_j)$ is a consistent estimator of $f(\omega_j)$.

When $p \to \infty$ as $T \to \infty$, we are interested in the conditional dependence structures of the pairs of coordinate, namely by defining $\Theta(\omega) = \bigl(f(\omega)\bigr)^{-1}$, our goal now is to provide a {\it sparse} estimator of $\Theta(\omega)$ with desirable large-sample properties.  Note that both $f(\omega)$ and $\Theta(\omega)$ are complex-valued matrices.  To make the following discussion easier, we first transform them into real-valued matrices.

For any $j = -\lfloor (n-1)/2, \ldots, \lfloor n/2\rfloor$ and $\omega_j = j/n$, since
\begin{align*}
\widetilde{f}_n(\omega_j) = &\frac{1}{(2M_n+1)n}\sum_{|k| \leq -M_n}\boldsymbol{d}(\omega_{j+k})\bigl(\boldsymbol{d}(\omega_{j+k})\bigr)^* \\
= & \frac{1}{(2M_n + 1)n}\sum_{|k| \leq M_n}\biggl\{\biggl(\sum_{t = 1}^n\boldsymbol{X}_t\cos(2\pi\omega_{j+k}t)\biggr)\biggl(\sum_{t = 1}^n\boldsymbol{X}_t\cos(2\pi\omega_{j+k}t)\biggr)^{\top}\\
& \hspace{2cm} + \biggl(\sum_{t = 1}^n\boldsymbol{X}_t\sin(2\pi\omega_{j+k}t)\biggr)\biggl(\sum_{t = 1}^n\boldsymbol{X}_t\sin(2\pi\omega_{j+k}t)\biggr)^{\top}\biggr\}\\
& \hspace{0.5cm} + \imath \frac{1}{(2M_n + 1)n}\sum_{|k| \leq M_n}\biggl\{\biggl(\sum_{t = 1}^n\boldsymbol{X}_t\cos(2\pi\omega_{j+k}t)\biggr)\biggl(\sum_{t = 1}^n\boldsymbol{X}_t\sin(2\pi\omega_{j+k}t)\biggr)^{\top}\\
& \hspace{2cm} - \biggl(\sum_{t = 1}^n\boldsymbol{X}_t\sin(2\pi\omega_{j+k}t)\biggr)\biggl(\sum_{t = 1}^n\boldsymbol{X}_t\cos(2\pi\omega_{j+k}t)\biggr)^{\top}\biggr\}\\
=: & A_j + \imath B_j,
\end{align*}
it follows from Lemma~\ref{lem-inverse} in the Appendix, that $\Theta(\omega_j)$ has the form $\widetilde{A}_j + \imath \widetilde{B}_j$, where $\widetilde{A}_j$ and $\widetilde{B}_j$ satisfy
\[
\left(\begin{array}{cc}
	A_j & -B_j \\
	B_j & A_j \\
\end{array}
\right)\left(\begin{array}{cc}
	\widetilde{A}_j & -\widetilde{B}_j \\
	\widetilde{B}_j & \widetilde{A}_j
\end{array}
\right) = I.
\]
Therefore, our problem is transformed to finding the inverse of $\biggl(\begin{array}{cc}A_j & -B_j \\ B_j & A_j \end{array}\biggr)$. %, where
%\begin{align*}
%A_j = \frac{1}{(2M_T + 1)T}\sum_{|k| \leq M_T}\sum_{s=1}^{T}\sum_{\ell=1}^T\boldsymbol{X}_s\boldsymbol{X}_{\ell}^{\top}\cos(2\pi\omega_{j+k}(s-l)),
%\end{align*}
%and
%\begin{align*}
%& B_j = -\frac{1}{(2M_T + 1)T}\sum_{|k|\leq M_T}\sum_{s=1}^{T}\sum_{\ell=1}^T\boldsymbol{X}_s\boldsymbol{X}_{\ell}^{\top}\sin(2\pi\omega_{j+k}(s-l)).
%\end{align*}

%Our problem is transferred to the following.  

Therefore, for any $j = -\lfloor (n-1)/2, \ldots, \lfloor n/2\rfloor$ and $\omega_j = j/n$, instead of directly studying $\widetilde{f}_n(\omega_j)$, our targets are now
\[
\Sigma_j := \left(\begin{array}{cc}
	\mathrm{Re}f(\omega_j) & \mathrm{Im}f(\omega_j) \\
	-\mathrm{Im}f(\omega_j) & \mathrm{Re}f(\omega_j)
\end{array}
\right)
\]
and sample version %$\Theta_{\Sigma_j} := (\Sigma_j)^{-1}$.  We try to find a sparse estimator $\widehat{\Theta}_{\Sigma_j}$ from
\begin{align*}
\widehat{\Sigma}_j = \frac{1}{(2M_n+1)n}\sum_{|k|\leq M_n}\sum_{s=1}^{n}\sum_{\ell=1}^n\left(\begin{array}{cc}
		\boldsymbol{X}_s\boldsymbol{X}_{\ell}^{\top}\cos(2\pi\omega_{j+k}(s-l)) & \boldsymbol{X}_s\boldsymbol{X}_{\ell}^{\top}\sin(2\pi\omega_{j+k}(s-l))\\
		- \boldsymbol{X}_s\boldsymbol{X}_{\ell}^{\top}\sin(2\pi\omega_{j+k}(s-l)) & \boldsymbol{X}_s\boldsymbol{X}_{\ell}^{\top}\cos(2\pi\omega_{j+k}(s-l))
	\end{array}\right).
\end{align*}

\subsection{Penalised precision matrices at every frequency point}\label{sec-22}

Now we have a sequence of expanded but real-valued smoothed periodogram matrices at every frequency point, i.e. $\bigl\{\widehat{\Sigma}_j, j = -\lfloor (n-1)/2 \rfloor, \ldots, \lfloor n/2 \rfloor\bigr\}$.  As for each one, our goal is to obtain a sparse inverse matrix.  In the last decade, a number of statistical methods have been proposed to achieve this goal, including graphical Lasso \citep[e.g.][]{YuanLin2007}, node-wise regression \citep[e.g.][]{MeinshausenBuhlmann2006}, constrained $\ell_1$-minimisation for inverse matrix estimation \citep[\textsc{clime}][]{CaiEtal2011}, adaptive \textsc{clime}\citep{CaiEtal2016} and the innovated scalable efficient estimation \citep{FanLv2016}, among others.  

In this paper, we do not intend to compare different sparse precision matrix estimation methods, but to apply the \textsc{clime} method for the sake of simplicity in technical details, and to provide with an example for consistent sparse precision matrix estimation in the high-dimensional frequency domain time series context.  For details of the \textsc{clime} method, we refer readers to \cite{CaiEtal2011}, which studies the inverse of the covariance matrices, and in which the sparse precision matrix estimators are obtained based on the sample covariance matrices of i.i.d. random vectors.  For completeness, we include the definition of the estimators.

For each $j \in \{ -\lfloor (n-1)/2 \rfloor, \ldots, \lfloor n/2 \rfloor \}$, let
	\begin{align}\label{eq-opt}
	\widehat{\Theta}_j = (\tilde{\Theta}_{j, kl}) = \argmin_{\|\widehat{\Sigma}_j \Theta_j - I\|_{\infty} \leq \lambda,\, \Theta_j \in \mathbb{R}^{2p \times 2p}} \|\Theta_j\|_1.
	\end{align}
	In practice, one can also symmetrise the estimator and obtain
	\[
	\widetilde{\Theta}_j = (\tilde{\Theta}_{j, kl}),
	\]
	where
	\[
	\tilde{\Theta}_{j, kl} = \tilde{\Theta}_{j, lk} = \hat{\Theta}_{j, kl} \mathbbm{1}\{\hat{\Theta}_{j, kl} \leq \hat{\Theta}_{j, lk}\} + \hat{\Theta}_{j, lk} \mathbbm{1}\{\hat{\Theta}_{j, lk} \leq \tilde{\Theta}_{j, kl}\}.
	\]

\section{Theory}\label{sec-theory}

%In this section, we are to show the uniform consistency of the sequence of precision matrices $\{\widehat{\Theta}_j\}$.  In Section~\ref{sec-assump}, the functional dependency \citep{Wu2005} framework is further discussed and other assumptions needed are introduced; the main results thereof are collected in Section~\ref{sec-results}.  All the proofs are in the Appendix.  %We are to present the uniform results for all $\omega$.  This is stronger than what one usually needs in practice, where the continuous sequence of $\omega$ is discretised.  It is worth to note that we provide results based on $\widehat{\Theta}$, which is defined in \eqref{eq-opt}.  The consistency results for $\widetilde{\Theta}$ follow automatically afterwards.

In \Cref{thm-key}, we will provide fixed-sample results for the spectral density matrix of a high-dimensional time series, in the form of an entry-wise error control between the smoothed periodogram estimator and the spectral density matrix.  This is a fundamental step in proving many different types of high-dimensional statistical problems.  To theoretically justify the sparse precision matrix estimator we proposed in \Cref{sec-method}, but more importantly, to demonstrate the power of \Cref{thm-key}, in \Cref{thm-main}, we show the uniform consistency of the sequence of precision matrices $\{\widehat{\Theta}_j\}$. 

As pointed out in  Section~\ref{sec-framework}, in order to provide the desired results, we are using the functional dependency framework described by \eqref{frame1} and \eqref{eq-theta-intro}.  To further characterise the dependency, we introduce Assumption~\ref{assump-c1}.  This is also used in \cite{ChenEtal2013}, and we refer interested readers there for examples.

\begin{assumption}\label{assump-c1}	
	Assume for some constant $0<\rho<1$, 
	\[
	\max_{i = 1, \ldots, p}\theta_{t, i} = O(\rho^{t}),
	\] 
	and for some constant $\kappa>0$ and $C_0 > 0$, 
	\[
	\max_{1\leq i\leq p}\mathbb{E}(\exp\{\kappa |X_{0,i}|\}) \leq C_{0}.
	\] 
%	In addition, $M_n \asymp n^{\beta}$ for some $0<\beta<1$.
\end{assumption}

Note that the fixed-sample result holds for all dimensionality, but in order to achieve desirable consistency results, we need extra conditions on the dimensionality of the data, which is detailed in \Cref{assump-c2}.  Note that we can actually handle a super-polynomial rate of $n$ for $p$, but in order to be specific, we assume $p$ is of any polynomial rate of $n$ as described in Assumption~\ref{assump-c2}.

\begin{assumption}\label{assump-c2} Assume:
	\begin{itemize}
	\item there exists constant $c > 0$ such that $p\leq cn^{r}$ for some $r>0$; 
	\item $M_n/T \to 0$, and there exists a constant $\delta > 0$ such that $M_n^{-1/2}(n/M_n)^{\delta} \to 0$.
	\end{itemize}
\end{assumption}

\Cref{assump-c3} is only  used to achieve the consistency of the sparse precision matrix estimators in \Cref{thm-main}.   Under \Cref{assump-c2}, \Cref{eq-assump-c3} holds even when the $\ell_1$- and $\ell_q$-norms of $\Theta_j$, $j = 1, \ldots, n$, diverge, as $n$ grows unbounded.  Therefore, \Cref{assump-c3} is a reasonably weak condition.

\begin{assumption}\label{assump-c3}
	Recall the parameter space $\mathcal{G}_q(c_{n, p}, M_{n, p})$ defined in \eqref{eq-sparsity-definition}.  Assume for $q \in [0, 1)$ the following holds:
		\begin{equation}\label{eq-assump-c3}
			M_{n, p}^{1-q}\left(\frac{M_{n, p}M_n}{n} + \frac{M_{n, p}n^{\delta}}{M_n^{1/2 + \delta}}\right)^{1-q}c_{n, p} = o(1).
		\end{equation}
\end{assumption}

\begin{thm}[Smoothed periodogram]\label{thm-key}
	Under Assumption~\ref{assump-c1}, there exists a constant $C > 0$ depending only on $\kappa$ and $C_0$ such that for any $\delta > 0$ and $H > 0$ the following holds
	\begin{align}\label{eq-thm-1-fixed}
		\mathbb{P}\left\{\sup_{k \in \{-\lfloor (n-1)/2 \rfloor, \ldots, \lfloor n/2 \rfloor\}} \max_{i, j = 1, \ldots, p} |\widetilde{f}_{ij,n}(\omega_k)-f_{ij}(\omega_k)| > C M_{n}/n+8(n/M_{n})^{1/2+\delta}n^{-1/2} \right\} \leq p^2n^{-H}.
	\end{align}

If we further assume Assumption~\ref{assump-c2}, then we have
	\[
		\sup_{k \in \{-\lfloor (n-1)/2 \rfloor, \ldots, \lfloor n/2 \rfloor\}} \max_{i, j = 1, \ldots, p} |\widetilde{f}_{ij,n}(\omega_k)-f_{ij}(\omega_k)|  = o_P(1).
	\]
\end{thm}

The fixed-sample result in \eqref{eq-thm-1-fixed} holds for any choices of sample size $n$, dimensionality $p$ and the smoothing window size $M_n$.  It holds in the functional dependency framework detailed in \Cref{assump-c1}, and provides an entry-wise error control of the smoothed periodogram and the spectral density matrix.  We adopt a union bound argument to handle the dimensionality and to provide a uniform result across the sampled frequency points.

It is worth mentioning that the probability upper bound allows for any $H > 0$, which allows for the dimensionality diverges at any arbitrary polynomial rate as the sample size diverges.  This is made explicit in \Cref{assump-c2}.

The detailed proof of \Cref{thm-key} is in the Appendix.  Here, we briefly outline the sketch of the proof.  We start with a fixed frequency point and a fixed entry in the matrix.  In order to bound the errors between the smoothed periodogram matrix $\widetilde{f}$ and the spectral density matrix $f$, we introduce a series of instrumental quantities, including an $m$-dependent series using conditional expectations, its truncated version which is truncated in magnitude by $(\log(n))^2$, and a centred version by subtracting the unconditional expectations.  The majority of the efforts are therefore dedicated to bound the differences of all these different quantities.  Applying triangle inequality yields desirable results for a fixed frequency point and a fixed entry in the matrix.  Finally, we apply a union bound argument to obtain \eqref{eq-thm-1-fixed}.

\begin{prop}\label{thm-main}
	Under Assumptions~\ref{assump-c1} and the parameter space defined in \eqref{eq-sparsity-definition}, for a constant $\delta > 0$, any $w \in [1, \infty]$ and
	\[
		\lambda \asymp \frac{M_{n, p}M_n}{n} + \frac{M_{n, p}n^{\delta}}{M_n^{1/2 + \delta}}, 
	\]
	we have for a sufficiently large constant $C > 0$,
	\begin{equation}\label{eq-thm-main-1}
		\mathbb{P}\left(\sup_{j \in \{-\lfloor (n-1)/2, \ldots, \lfloor n/2 \rfloor\}} \|\hat{\Theta}(\omega_j) - \Theta(\omega_j)\|_w \leq CM_{n, p}^{1-q}\lambda^{1-q}c_{n, p}\right) \geq 1 - p^2n^{-H}.
	\end{equation}
	
	If we further assume Assumptions~\ref{assump-c2} and \ref{assump-c3}, then we have
	\[
		\sup_{j \in \{-\lfloor (n-1)/2, \ldots, \lfloor n/2 \rfloor\}} \|\hat{\Theta}(\omega_j) - \Theta(\omega_j)\|_w = o_P(1).
	\]
\end{prop}

\Cref{thm-main} is an application of \Cref{thm-key} on the sparse precision matrix estimation.  The proof is in fact straightforward based on \eqref{eq-thm-1-fixed} and the proof techniques developed in \cite{CaiEtal2011}.  Since it is built upon \Cref{thm-key}, we allow for the same flexibility that in \eqref{eq-thm-main-1}, $H$ is allowed to be any positive value, and therefore the dimensionality $p$ is allowed to be of any arbitrary order of the sample size $n$.

%=================
%\section{Numerical Study}

%\subsection{Tuning parameters}

%As a penalisation method, tuning parameter selection is important but tricky, and dealing with a frequency domain problem makes it even trickier.  The difficulties here are twofold: 
%\begin{itemize}
%	\item [a)] If we want to use an information-type criterion, then we need a suitable likelihood-type loss function.  One natural candidate is Whittle likelihood.  However there are two problems here.  Firstly, the approximation of Whittle likelihood is high-dimensional regime is unjustified; secondly, the expression of Whittle likelihood involves estimators at every single frequency point.  The latter results in a loop, the convergence of which is unjustified, and which is computationally expensive.
%	\item [b)] If we rule out the information-type criterion, then the only choice left for us is data-driven methods based on subsampling.  This again results in another difficulty - the DFTs are different frequency point are correlated with too complicated finite sample correlation structure to be subsampled from.
%\end{itemize}

%Having said these, in the numerical demonstrations in this section, we

%\subsection{Further smoothing}

%=================
\section{Numerical Results}\label{sec-numerical}

\subsection{Simulations}

In this section, we verify our proposed methodology using simulated data. We consider multivariate time series having dimension $p = 10$ or $50$ with sample size $n = 200$ or $400$. These are challenging scenarios for spectral analysis because the amount of data available to estimate the spectral density matrix and its inverse is related to the smoothing span $2M_n + 1$ used to smooth the periodogram matrix, and not the length of the time series. In our simulations, we picked $M_n$ using the generalised cross-validation (GCV) criterion developed by \cite{Ombao01}.  Using this approach to pick $M_n$, we also construct the smoothed periodogram matrix $\tilde{f}_n(\omega)$ and calculate its inverse (whenever possible) and use these estimators in order to assess relative performance.

We investigated multiple scenarios in this study: we simulated from (1) a $p$-variate Gaussian white noise model, (2) a $p$-variate first-order vector autoregressive (VAR(1)) model, whose parameters we give below, and (3) a $p$-variate VAR(1) model whose conditional dependence structure between the dimensions is driven by a sparse precision matrix of the innovations.  

Setting (1) allows us to see how our methodology performs relative to the smoothed periodogram matrix in a very simple scenario where the spectral density matrix and its inverse do not change across frequencies, which allow us to evaluate relative performance only as a function of dimensionality. Setting (2) allows us to see how our methodology performs when the data exhibit some degree of autocorrelation and lagged cross-correlation. To construct the VAR(1) model, we set the $p \times p$ coefficient matrix to be a banded matrix such with diagonal entries set to be $0.5$, and for the $j$th row, $j \in \{1, \ldots, p-2\}$, we set the $(j+1)$th column to be $-0.3$ and the $(j+2)$th column to be $0.2$. We use the identity matrix as the covariance matrix for the innovations in the model. Setting (3) creates heterogeneity in the marginal variances, and hence, in the diagonal elements of the spectral density matrix, but truth has a sparse conditional dependence structure.  In particular, we let the VAR(1) coefficient matrix be a diagonal matrix with entries randomly selected from the interval $(0.25, 0.75)$, and a random sign.  The precision matrix for the innovations vector is sparse, with off-diagonal elements equal to 0 or $0.5$ with probability $0.5$.

We evaluate performance in the following ways. First, we use the mean integrated squared error (MISE), defined by
\[
	\mathrm{MISE}\bigl(\bigl\{\widehat{\Theta}(\omega_j)\bigr\}_{j=1}^n, \bigl\{\Theta(\omega_j)\bigr\}_{j=1}^n\bigr) = \frac{2}{n}\sum_{j=1}^{n/2} \bigl\|\widehat{\Theta}(\omega_j) - \Theta(\omega_j)\bigr\|_{*}^2,
\]
where $\{\omega_j\}_{j=1}^{n/2}$ denote the Fourier frequencies in the interval $(0, 0.5)$, $\|\cdot\|_{*}$ denotes the Frobenius norm of a matrix but discarding the diagonal entries, i.e. for a matrix $A = (A_{ij}) \in \mathbb{R}^{p \times p}$,
	\[
		\|A\|_* = \sqrt{\sum_{i = 1}^p \sum_{j \neq i}A_{ij}^2}.
	\] 
	The reason we are discarding the diagonal entries is that we are mainly interested in the off-diagonal entries, and the penalisations deployed in obtaining the sparse precision matrix estimators inevitably introduce bias, especially for the diagonal entries.  If one would like a better estimator of the diagonal entries, one can adopt an optional second step updating the diagonal entries only by forcing the product of the smoothed periodogram matrix and the sparse precision matrix to be identity.  Due to the lack of theoretical guarantees, we omit this optional step in this paper.

We compare our estimator (SIPE) to the na\"{i}ve inverse of the smoothed periodogram matrix (Na\"{i}ve), with smoothing span being the modified Daniell kernel with bandwidth picked using the GCV criterion, and the shrinkage estimator (Shrinkage) by \cite{Bohm:vonSachs:2009}.  We collect the numerical results averaged over 50 repetitions in each setting in Table \ref{sim:miseInverse}.  Each cell of the table is of the form mean (standard deviation).  Since the Na\"{i}ve estimator and the Shrinkage estimator do not produce sparse estimation, we only report the evaluations on the support recovery for the SIPE.  We define the true positive proportion (TPP) and true negative proportion (TNP) as follows.
	\begin{align*}
		& \mathrm{TPP} = \frac{\# \mbox{non-zero diagonal entries in the estimator}}{\# \mbox{non-zero diagonal entries in the truth}}, \\
		& \mathrm{TNP} = \frac{\# \mbox{zero diagonal entries in the estimator}}{\# \mbox{zero diagonal entries in the truth}}.
	\end{align*}
	The results reported are averaged across all frequencies.

First, looking across all simulation settings, we see that the smoothed periodogram matrix sometimes cannot be inverted, motivating the need for some type of regularisation. The spectral density matrix for the white noise (WN) model is the identity matrix across all frequencies. The Shrinkage is biased towards a scaled identity matrix, hence its superior performance in this setting for all dimensionalities and sample sizes. When the time series data possess autocorrelation, such as in the VAR(1) and sparse VAR(1) (sVAR(1)) models, SIPE is competitive with the shrinkage estimator with respect to MISE, yet can reasonably estimate the zero and non-zero entries of the precision matrices. In contrast, the shrinkage estimator behaves like a ridge estimator, and hence, by construction cannot obtain sparse estimates of the inverse spectral density matrix.  We see that our estimator yields favourable estimates of the spectral precision matrix while giving relatively good estimates on which entries of the spectral precision matrix are truly zero or non-zero.

\begin{table}
\centering
\begin{tabular}{ccc|rrr| rr}
Simulation & & & \multicolumn{3}{c}{MISE - Precision Matrix} & \multicolumn{2}{|c}{SIPE}\\ \cline{4-6} \cline{7-8}
Setting & $p$ & $T$ & Na\"{i}ve & Shrinkage & SIPE & TPP & TNP \\
\hline
WN & 10 & 200 & 21.62 (31.58) & 0.39 (1.23) & 0.17 (0.69) & 0.83 (0.05) & 0.82 (0.05) \\
& & 400 & 13.17 (21.24) & 0.22 (0.62) & 0.60 (1.91) & 0.74 (0.05) & 0.74 (0.05) \\
& 50 & 200 & 202.44 (258.17) & 0.04 (0.01) & 0.54 (1.62) & 0.71 (0.01) & 0.69 (0.01) \\
& & 400 & 13.05 (20.57) & 0.02 (0.01) & 1.06 (3.11) & 0.62 (0.01) & 0.62 (0.01) \\
VAR(1) & 10 & 200 & 16.36 (47.10) & 3.76 (3.93) & 3.60 (0.02) & 0.91 (0.02) & 0.90 (0.03) \\
& & 400 & 8.94 (18.89) & 3.69 (6.13) & 3.60 (0.01) & 0.89 (0.02) & 0.87 (0.02) \\
& 50 & 200 & - & 3.62 (0.44) & 4.18 (3.85) & 0.86 (0.01) & 0.81 (0.02) \\
& & 400 & - & 3.57 (0.55) & 3.71 (1.04) & 0.86 (0.01) & 0.84 (0.01) \\
sVAR(1) & 10 & 200 & 119.64 (230.51) & 12.57 (19.40) & 10.18 (20.18) & 0.97 (0.02) & 0.97 (0.03) \\
 & & 400 & 47.04 (99.37) & 12.87 (25.12) & 9.90 (19.62) & 0.94 (0.03) & 0.94 (0.03) \\
 & 50 & 200 & - & 17.11 (24.30) & 15.97 (23.27) & 0.97 (0.01) & 0.96 (0.02) \\
 & & 400 & - & 14.31 (20.97) & 15.96 (23.27) & 0.95 (0.02) & 0.94 (0.02) \\
\end{tabular}
\caption{Simulation results for estimating the inverse of the spectral density matrix, with mean integrated squared error (MISE), and true positive proportions (TPP) and true negative proportions (TNP).  All results entries are in the form of mean (standard deviation). Hyphenated entries (-) denote that the smoothed periodogram matrix could not be inverted.  TPP and TNR are reported for SIPE only. MISE entries were multiplied by $10^3$ for clarity. }
\label{sim:miseInverse}
\end{table}

%\begin{table}
%\centering
%\begin{tabular}{cccccc}
%Simulation & & & \multicolumn{3}{c}{MISE - Precision Matrix} \\ \cline{4-6}
%Setting & $p$ & $T$ & Sm. Periodogram & Shrinkage & SIPE \\
%\hline
%White Noise & 10 & 200 & 2.64 & 0.05 & 2.18 \\
%& & 400 & & & \\ 
%& 50 & 200 & 887.91& & 27.22 \\
%& & 400 & & &  \\
%VAR(1) & 10 & 200 & 353.37 & 324.17 & 423.06 \\
%& & 400 & & & \\
%& 50 & 200 & & 22.76 & 29.13 \\
%& & 400 & & & 
%\end{tabular}
%\caption{Simulation results for estimating the spectral density matrix.}
%\label{sim:mise}
%\end{table}

\subsection{Analysis of the Google Flu Trends Data}
We give an empirical illustration of our proposed methodology by analysing the Google Flu Trends data set. Researchers at Google used select Google search terms to predict influenza activity \citep{Ginsberg09}. The resulting data set consists of weekly predicted numbers of influenza-like-illness related visits out of every 100,000 random outpatient visits within select cities throughout the United States of America. The data set is further aggregated at the state-level and region-level, where the latter comprises of different states. The version of the Google Flu Trends data set we used is the state-level aggregate of log-transformed weekly data from 1 January 2006 to 6 October 2013. The resulting time series thus has $p=50$ dimensions and length $n = 406$.

The goal of our analysis is to investigate the conditional dependencies of the time series across states. To this end, we need to estimate the partial coherence matrix, which is a function of the inverse of the spectral density matrix. The partial coherence matrix is the frequency domain analog of partial correlation, and can be interpreted as the correlation between two time series that have been bandpass filtered at frequency $\omega$, after removing the linear effects of the other time series. The $(j,k)$th element of the partial coherence matrix is $\rho_{jk}(\omega) = |\Theta_{jk}(\omega)|^2/[\Theta_{jj}(\omega)\Theta_{kk}(\omega)]$, where $\Theta(\omega)$ is the inverse spectral density matrix. We use our methodology to obtain a sparse estimate of $\Theta(\omega)$, from which we can then obtain estimates of partial coherence. We are only interested in the partial coherence matrix, and so we centre each time series to have mean zero and then we standardised them to have unit variance.

To pick the parameters of our method, we choose $M_n$ using the GCV criterion. Each of the fifty time series were driven by frequencies within the frequency band $(0, 0.10)$, as shown by the diagonal entries of $\widetilde{f}_{n}$ in Figure~\ref{dat:power}. Indeed, for each of the fifty time series, the variance attributed to each Fourier frequency outside of this band is less than 5\% of the overall variation. Thus, we estimate the partial coherence within this frequency band, and we further summarise our results by taking the median partial coherence within this frequency band. We show our results in Figure~\ref{dat:pcoh}.

\begin{figure}[!htbp]
    \begin{minipage}{0.35\textwidth}
        \centering
			\includegraphics[width=0.8\textwidth,keepaspectratio]{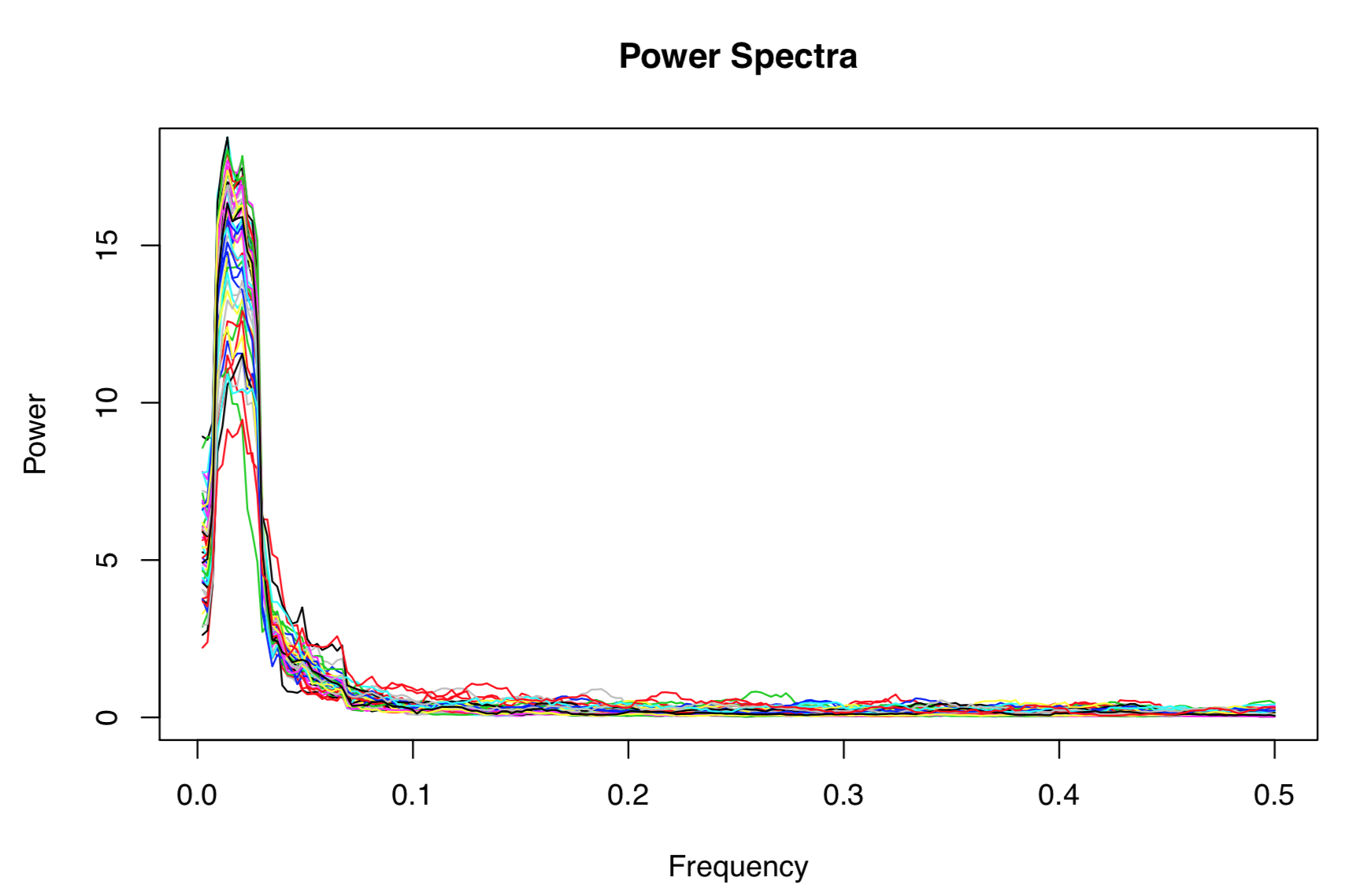}		
			\caption{Power spectra for each state's time series in the Google Flu Trends data. Each colour denotes the power spectrum for one state.} \label{dat:power}
    \end{minipage}\hfill
    \begin{minipage}{0.55\textwidth}
        \centering
        \includegraphics[width=0.4\textwidth,keepaspectratio]{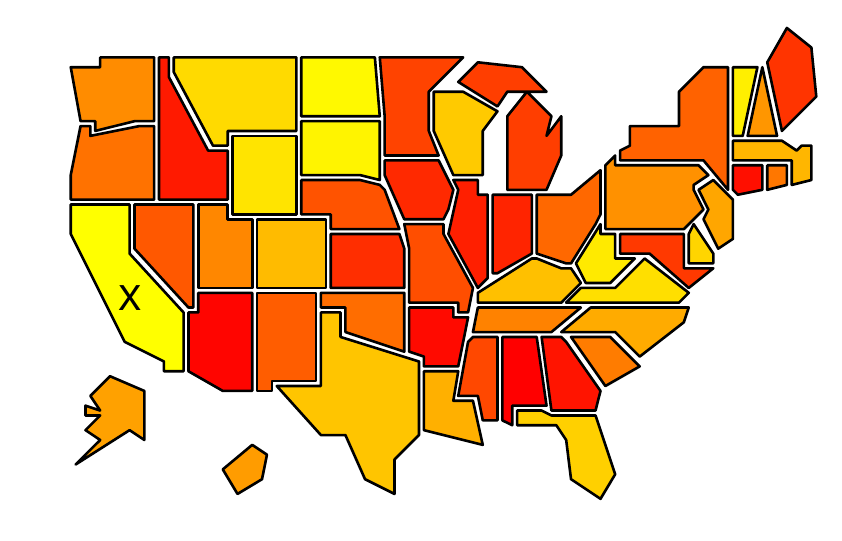}		
\includegraphics[width=0.4\textwidth,keepaspectratio]{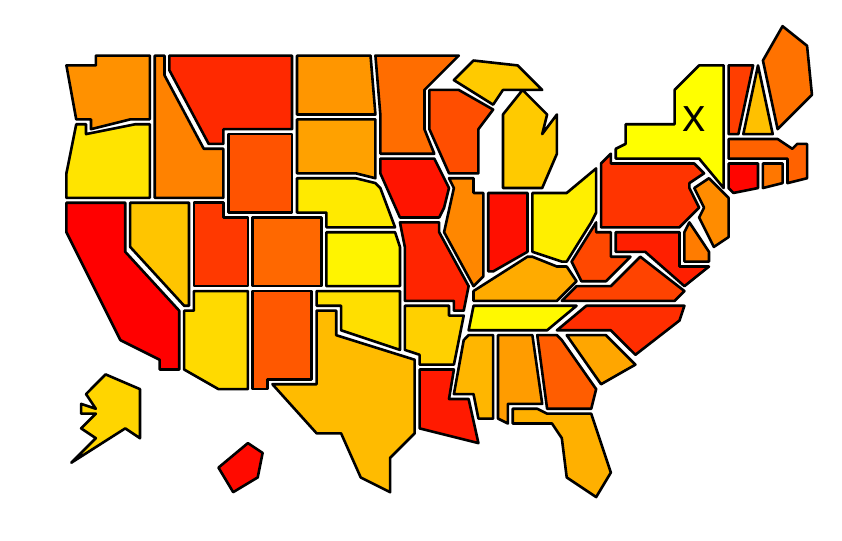}
\includegraphics[width=0.4\textwidth,keepaspectratio]{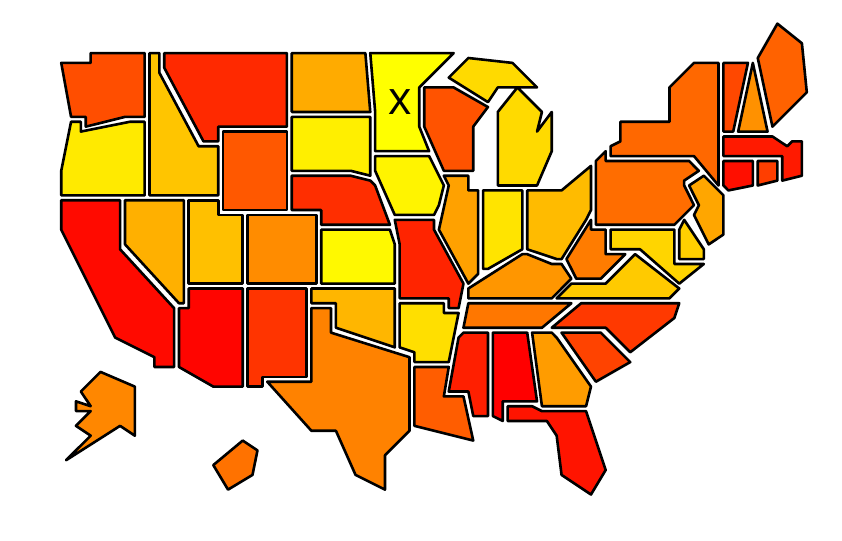}
\includegraphics[width=0.4\textwidth,keepaspectratio]{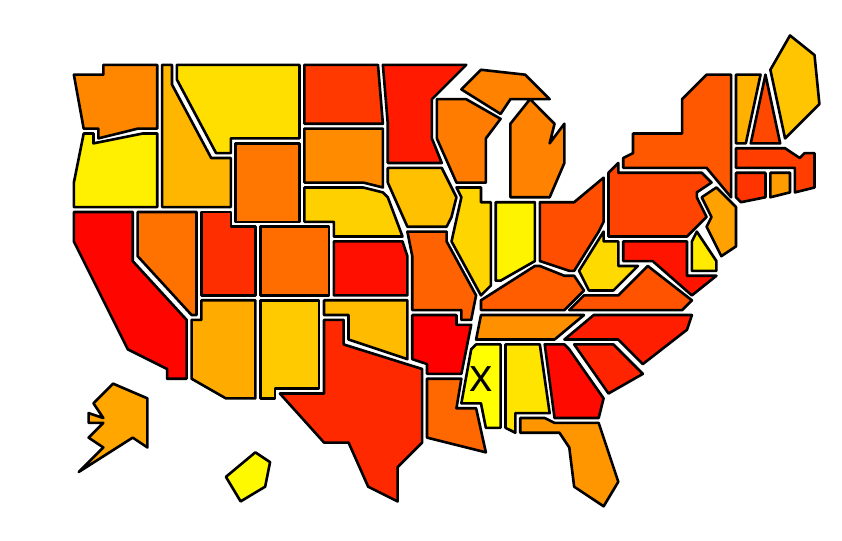}		
\caption{The conditional dependence structure of the Google Flu Trends data across the fifty states between one state marked by an `X' and the other forty-nine states.  In the top-left, top-right, bottom-left and bottom-right, the `X'es are California, New York, Minnesota and Mississippi, respectively.  Yellow and red colours indicate high and low values of partial coherence, respectively.}
\label{dat:pcoh}
    \end{minipage}
\end{figure}

Each of four geographically distinct states (California, New York, Minnesota and Mississippi) yields different conditional independencies. First, we see a local spatial structure. For instance, we see conditional dependencies between Minnesota and its neighbouring Midwest states, % (or states within HHS Region 5), 
and conditional dependencies between Mississippi and Alabama, Florida, and Tennessee. % (or states within HHS Region 4). 
On the other hand, we also a long-distance structure, e.g., New York with states including Oregon and Washington. Previous analyses have yielded similar results \citep{Davis16}. Looking at all pairwise conditional dependencies yielded a sparse partial coherence matrix, where 28.2\% of the pairs were 0.  
% Our results show conditional dependencies between states in Region 1 (ME, VT, RI), with both ME and VT serving as points for cross-region conditional dependencies (MD in Region 3; WI in Region 5; IA in Region  7; MT in Region 8). Other within-region conditional dependencies include AL and GA, both in Region 4, MN and WI, both within Region 5, and MT and SD, both in Region 8. All other conditional dependencies are between states belonging to different regions.

\section*{Appendix}\label{sec-app}

In this section, we collect all the necessary technical details.

\begin{lemma}\label{lem-inverse}
Let $Z := A + \imath B \in \mathbb{C}^{p \times p}$, with $A, B \in \mathbb{R}^{p \times p}$.  Assume $Z$ is non-singular and the inverse of $Z$ is denoted as $Z^{-1}$, then $Z^{-1} = \widetilde{A} + \imath \widetilde{B}$, where $\widetilde{A}, \widetilde{B}\in \mathbb{R}^{p \times p}$, satisfying
\[
\left(\begin{array}{cc}
	A & -B \\
	B & A \\
\end{array}
\right)\left(\begin{array}{cc}
	\widetilde{A} & -\widetilde{B} \\
	\widetilde{B} & \widetilde{A}
\end{array}
\right) = I_{\mathrm{2p}}.
\]
\end{lemma}

\begin{proof}
It follows from the fact $ZZ^{-1} = I_{\mathrm{p}}$ that 
\[
I_{\mathrm{p}} = (A + \imath 	B)(\widetilde{A} + \imath\widetilde{B}) = (A\widetilde{A}-B\widetilde{B}) + \imath(B\widetilde{A}+A\widetilde{B}),
\]	
which is equivalent to 
\begin{align*}
\left(\begin{array}{cc}
	A & -B \\
	B & A \\
\end{array}
\right)\left(\begin{array}{c}
	\widetilde{A}\\
	\widetilde{B}
\end{array}
\right) = \left(
\begin{array}{c}
	I_{\mathrm{p}} \\
	0
\end{array}
\right),
\, \mbox{and }
\left(\begin{array}{cc}
	A & -B \\
	B & A \\
\end{array}
\right)\left(\begin{array}{c}
	-\widetilde{B}\\
	\widetilde{A}
\end{array}
\right) = \left(
\begin{array}{c}
	0 \\
	I_{\mathrm{p}}
\end{array}
\right).
\end{align*}
Therefore,
\[
\left(\begin{array}{cc}
	A & -B \\
	B & A \\
\end{array}
\right)\left(\begin{array}{cc}
	\widetilde{A} & -\widetilde{B} \\
	\widetilde{B} & \widetilde{A}
\end{array}
\right) = I_{\mathrm{2p}}.
\]

\end{proof}

\begin{proof}[Proof of \Cref{thm-key}]

This proof starts with proving the result for any fixed $\omega$.  For any $(i, j) \in \{1, \ldots, p\}^{\otimes 2}$, note that the $(i, j)$ entry of the periodogram $P_n(\omega)$ can be written as
	\begin{align}
		P_{n, ij}(\omega) = &\frac{1}{n}\sum_{t=1}^{n}X_{t,i}\exp(-\imath2\pi\omega t)\sum_{t=1}^{n}X_{t,j}\exp(\imath2\pi\omega t) \nonumber \\
		= &\frac{1}{n}\sum_{t=2}^{n}\sum_{l=1}^{t-1}X_{t,i}X_{l,j}\exp(-\imath2\pi\omega (t-l)) + \frac{1}{n}\sum_{l=2}^{n}\sum_{t=1}^{l-1}X_{t,i}X_{l,j}\exp(-\imath2\pi\omega (t-l)) \nonumber\\
		& \hspace{1cm} + \frac{1}{n}\sum_{t=1}^{n}X_{t,i}X_{t,j} \nonumber \\
		=: & P^{(1)}_{n, ij}(\omega)+P^{(2)}_{n, ij}(\omega)+P^{(3)}_{n, ij}. \label{eq-thm3-step1}
	\end{align}
 
Next, we are to bound the three terms in the right-hand side of \eqref{eq-thm3-step1} separately.  As for the term $(I)$, we will approximate it by a similar quantities built up by $m$-dependent random variables.   Let 
	\begin{eqnarray*}
		\tilde{f}^{(1)}_{n, ij}(\omega):=\frac{1}{2M_{n}+1}\sum_{s=-M_{n}}^{M_{n}} P^{(1)}_{n, ij}(\omega+s/n)=\frac{1}{n}\sum_{t=2}^{n}\sum_{l=1}^{t-1}X_{t,i}X_{l,j}a_{t-l}(\omega),
	\end{eqnarray*}
	where
	\begin{align}
		a_{k}(\omega) & =\frac{1}{2M_{n}+1}\sum_{s=-M_{n}}^{M_{n}}\exp(-\imath2\pi k (\omega+s/n)) = \frac{1}{2M_n+1} \exp(-\imath2\pi k\omega) \sum_{s = -M_n}^{M_n} \exp(-\imath 2\pi ks/n) \nonumber \\
		& = \frac{1}{2M_n+1} \exp(-\imath2\pi k\omega) \sum_{s = -M_n}^{M_n} \cos (2\pi ks/n) =\exp(-\imath2\pi k\omega)\frac{\sin(2\pi(M_{n}+1/2)k/n)}{(2M_{n}+1)\sin (\pi k/n)}. \label{eq-akomega}
	\end{align}
	The last identity in \eqref{eq-akomega} follows from the trigonometric identity that for any $\alpha$, which is not a multiple of $2\pi$, and any positive integer $m \in \mathbb{N}_+$, we have
	\[
		\sum_{k=0}^m \cos(\phi + k\alpha) = \sin((m+1)\alpha/2)\cos(\phi + m\alpha/2)/\sin(\alpha/2).
	\]
 
In addition, for any $\alpha$, which is not a multiple of $\pi$, and any $m \in \mathbb{N}_+$, it holds that
 	\[
 		\left|\frac{\sin(m\alpha)}{\sin(\alpha)}\right| = \left|\frac{\exp(\imath m\alpha)\sin(m\alpha)}{\exp(\imath \alpha)\sin(\alpha)}\right| = \left|\frac{1 - \exp(\imath 2m\alpha)}{1 - \exp(\imath 2\alpha)}\right| = \left|\sum_{n = 0}^{m-1} \exp(\imath 2\alpha)\right| \leq m.
 	\]
	Then for $k \in \{1, \ldots, n-1\}$,	 we have
	\begin{equation}\label{yi-1}
		\frac{\sin^{2}(2\pi(M_{n}+1/2)k/n)}{(2M_{n}+1)^{2}\sin^{2} (\pi k/n)}\leq 1.
	\end{equation}
	For $k \in \{\lfloor n/(M_{n}+1/2)\rfloor +1, \ldots, \lfloor n-n/(M_{n}+1/2)\rfloor \}$, we have
	\begin{equation}\label{yi-2}
		|\sin(\pi k/n)|\geq \pi/2\min(k, n-k)/n.
	\end{equation}
	Combining \eqref{yi-1} and \eqref{yi-2} we have
	\begin{align*}
		\sum_{k=1}^{n-1}|a_{k}(\omega)|^{2} =& \sum_{k=1}^{n-1}\frac{\sin^{2}(2\pi(M_{n}+1/2)k/n)}{(2M_n+1)^2\sin^{2} (\pi k/n)} \\
		=& \left(\sum_{k=1}^{\lfloor n/(M_n+1/2)\rfloor} + \sum_{k = \lfloor n/(M_n+1/2)\rfloor + 1}^{\lfloor n - n/(M_n+1/2)\rfloor} + \sum_{k = \lfloor n - n/(M_n+1/2)\rfloor + 1}^{n-1}\right) \frac{\sin^{2}(2\pi(M_{n}+1/2)k/n)}{(2M_n+1)^2\sin^{2} (\pi k/n)}\\
		\leq & \min \left\{2\lceil n/(M_n+1/2)\rceil + \frac{\pi^2}{3}\frac{n^2}{(2M_n+1)^2}, \, n-1\right\},
	\end{align*}
	where the last inequality follows from the fact that $\sum_{k=1}^{\infty} = \pi^2/6$.  Then,
	\[
		\sum_{k=1}^n|a_{k}(\omega)|^{2} \leq \min \left\{2\lceil n/(M_n+1/2)\rceil + \frac{\pi^2}{3}\frac{n^2}{(2M_n+1)^2} + 1, \, n\right\} =: A_n;
	\] 
	and
	\begin{align}\label{eq-maxak}
		\max_{k = 1, \ldots, n}|a_k(\omega)| = 1,
	\end{align}
	by noting that $|a_n(\omega)| = 1$.

For $i = 1, \ldots, p$, let $\bar{X}_{t,i}=\mathbb{E} (X_{t,i}|\mathcal{F}_{t,m})$,  where $\mathcal{F}_{t, m}=(\boldsymbol{e}_{t-m},\ldots,\boldsymbol{e}_{t})$ with $m= \lceil (\log(n))^{2}\rceil $.  Note that
	\begin{align*}
		X_{t, i} - \bar{X}_{t,i} = \sum_{j = m+1}^{\infty} \bigl(\mathbb{E}(X_{t,i} \mid \mathcal{F}_{t,j}) - \mathbb{E}(X_{t, i} \mid \mathcal{F}_{t, j-1})\bigr) =: \sum_{j=m+1}^{\infty} d_{j},
	\end{align*}
	with $\bigl\{\mathbb{E}(d_j^2)\bigr\}^{1/2} \leq \theta_{t,j}$.  It follows from Assumption~\ref{assump-c1} and Theorem~1(ii) in \cite{Wu2005}, that
	\begin{equation}\label{eq-barnobar}
		\bigl\{\mathbb{E}((\bar{X}_{t,j} - X_{t,j})^2)\bigr\}^{1/2} = O(\rho^m).
	\end{equation}
	Define 
	\begin{eqnarray*}
		\bar{f}^{(1)}_{ij,n}(\omega)=\frac{1}{n}\sum_{t=2}^{n}\sum_{l=1}^{t-1}\bar{X}_{t,i}\bar{X}_{l,j}a_{t-l}(\omega).
	\end{eqnarray*}
	It follows from Proposition~1 in \cite{LiuWu2010}, \eqref{eq-maxak} and \eqref{eq-barnobar} that 
	\begin{equation}\label{eq-firstterm}
		\mathbb{E}|\bar{f}^{(1)}_{ij,n}(\omega)-\tilde{f}^{(1)}_{ij,n}(\omega)|=O(n\rho^{m}). 
	\end{equation}

Now let $M=(\log(n))^{2}$ and define
	\begin{align*}
		& \breve{X}_{t,i}=\bar{X}_{t,i}\mathbbm{1}\{|\bar{X}_{t,i}|\leq M\}, \quad \hat{X}_{t,i}=\breve{X}_{t,i}-\mathbb{E} (\breve{X}_{t,i}), \\
		& \breve{f}^{(1)}_{ij,n}(\omega)=\frac{1}{n}\sum_{t=2}^{n}\sum_{l=1}^{t-1}\breve{X}_{t,i}\breve{X}_{l,j}a_{t-l}(\omega), \quad \hat{f}^{(1)}_{ij,T}(\omega)=\frac{1}{n}\sum_{t=2}^{n}\sum_{l=1}^{t-1}\hat{X}_{t,i}\hat{X}_{l,j}a_{t-l}(\omega).
	\end{align*}
	Note that for centred random vectors, we have 
	\begin{align}
		& |\mathbb{E}(\bar{X}_{t,i}\mathbbm{1}\{|\bar{X}_{t,i}|\leq M\})|=|\mathbb{E} (\tilde{X}_{t,i}\mathbbm{1}\{|\tilde{X}_{t,i}|> M\})| \leq \mathbb{E}\bigl\{|\bar{X}_{t,i}|\mathbbm{1}\{|\bar{X}_{t,i}| > M\}\bigr\} \nonumber\\
		\leq & \int_M^{\infty} \mathbb{P}\bigl\{|\bar{X}_{t,i}| > t\bigr\}\, dt \leq C_0\int_M^{\infty} \exp(-\kappa t)\,dt = C_0\kappa^{-1}\exp(-\kappa M), \label{eq-secondterm}
	\end{align}
	where the last inequality follows from Assumption~\ref{assump-c1} and Markov's inequality.

It also follows from Assumption~\ref{assump-c1} that
	\begin{equation}\label{eq-thirdterm}
		\mathbb{P}(\bar{f}^{(1)}_{ij,n}(\omega)\neq \breve{f}^{(1)}_{ij,n}(\omega))\leq \max \Bigl\{\sum_{t=1}^{n}\mathbb{P}(|X_{t,i}|\geq M),\, \sum_{t=1}^{n}\mathbb{P}(|X_{t,j}|\geq M) \Bigr\} \leq 2C_0n\exp(-\kappa M).
	\end{equation}
	
To this end, we have for any $\varepsilon > 0$,
	\begin{align}
		& \mathbb{P}\bigl\{|\tilde{f}^{(1)}_{ij, n}(\omega) - \hat{f}^{(1)}_{ij,n}(\omega)| > \varepsilon \bigr\}	\nonumber \\
		\leq & \mathbb{P}\bigl\{|\tilde{f}^{(1)}_{ij, n}(\omega) - \bar{f}^{(1)}_{ij,n}(\omega)| > \varepsilon/3 \bigr\}  + \mathbb{P}\bigl\{|\bar{f}^{(1)}_{ij, n}(\omega) - \breve{f}^{(1)}_{ij,n}(\omega)| > \varepsilon/3 \bigr\} + \mathbb{P}\bigl\{|\breve{f}^{(1)}_{ij, n}(\omega) - \hat{f}^{(1)}_{ij,n}(\omega)| > \varepsilon/3 \bigr\} \nonumber \\
		=: & (I) + (II) + (III). \label{eq-decomp}
	\end{align}

Moreover, it follows from Markov's inequality and \eqref{eq-firstterm} that
	\begin{equation}\label{eq-decomp1}
		(I) \leq \frac{\mathbb{E}\bigl(|\tilde{f}^{(1)}_{ij, n}(\omega) - \bar{f}^{(1)}_{ij,n}(\omega)|\bigr)}{\varepsilon/3} = O(n\rho^m \varepsilon^{-1}).
	\end{equation}
	It follows from \eqref{eq-thirdterm} that
	\begin{equation}\label{eq-decomp2}
		(II) \leq \mathbb{P}(\bar{f}^{(1)}_{ij,n}(\omega)\neq \breve{f}^{(1)}_{ij,n}(\omega)) \leq 2nC_0 \exp\{-\kappa M\}.
	\end{equation}
	Due to Markov's inequality and \eqref{eq-secondterm}, the following holds
	\begin{align}\label{eq-decomp3}
		(III) = O\bigl(n\varepsilon^{-1}\exp\bigl(-2\kappa M\bigr)\bigr).
	\end{align}
 	Finally, combining \eqref{eq-decomp}, \eqref{eq-decomp1}, \eqref{eq-decomp2} and \eqref{eq-decomp3}, we obtain that
 	\begin{equation}\label{eq-tildehatdiff}
 		\mathbb{P}\bigl\{|\tilde{f}^{(1)}_{ij, n}(\omega) - \hat{f}^{(1)}_{ij, n}(\omega)| > \varepsilon \bigr\} = O\bigl(n\rho^m \varepsilon^{-1} + n\exp\{-\kappa M\} + n\varepsilon^{-1}\exp\bigl(-2\kappa M\bigr)\bigr).
 	\end{equation}
 
\vskip 3mm 

Note that $(\hat{X}_{t,i},\hat{X}_{t,j})$, $1\leq t\leq n$ are also $m$-dependent random vectors with zero means.  In addition, we have \eqref{eq-maxak},
	\[
		\max_{\stackrel{t = 1, \ldots, n}{i = 1, \ldots, p}} \mathbb{E}(\hat{X}^2_{t, i}) \leq K_0, \quad \max_{\stackrel{t = 1, \ldots, n}{i = 1, \ldots, p}} \mathbb{E}(\hat{X}^4_{t, i} )\leq K_0,
	\]
	where $K_0$ only depends on $C_0$ following from Assumption~\ref{assump-c1}.  Therefore it follows from Proposition~3 in \cite{LiuWu2010} that for any $x\geq 1$, $y\geq 1$ and any constant $Q>0$ we have,
	\begin{align*}
		& \mathbb{P}(|\hat{f}^{(1)}_{ij, n}(\omega) - \mathbb{E}\bigl(\hat{f}^{(1)}_{ij, n}(\omega)\bigr)|\geq x/n)\\
		\leq &  2e^{-y/4}+C_{1}n^{3}M^{2}\Big{(}x^{-2}y^{2}m^{3}(M^{2}+n)\sum_{k=1}^{n}|a_{k}(\omega)|^{2}\Big{)}^{Q} \\
		& \hspace{0.5cm}+ C_{1}n^{4}M^{2}\max\left\{\mathbb{P}\left(|\hat{X}_{0,i}|\geq\frac{C_{2}x}{ym^{2}(M+n^{1/2})}\right), \, \mathbb{P}\left(|\hat{X}_{0,j}|\geq\frac{C_{2}x}{ym^{2}(M+n^{1/2})}\right)\right\},
	\end{align*}
	where $C_{1}$ and $C_{2}$ are  positive constants depending only on $Q$, $\kappa$ and $C_{0}$. 

For any $\delta > 0$ and $H > 0$, letting $x=(n/M_{n})^{1/2+\delta}n^{1/2}$ and $y=(\log(n))^{2}$, there exists a constant $C_3 > 0$  only depending on $H$, $\kappa$ and $C_0$, such that
	\begin{align}\label{eq-LiuWu2010Prop3}
		& \mathbb{P}(|\hat{f}^{(1)}_{ij,n}(\omega) - \mathbb{E}\bigl(\hat{f}^{(1)}_{ij,n}(\omega)\bigr)|\geq (n/M_{n})^{1/2+\delta}n^{-1/2}) \nonumber\\
		\leq & 2n^{-\log (n)/4} + C_3 n^{3-(1+2\delta)Q} A_n^Q M_n^{(1+2\delta)Q}(\log(n))^{4 + 10Q} \nonumber \\
		\leq & C_3 n^{-H}.
	\end{align}
	Combining \eqref{eq-tildehatdiff} and \eqref{eq-LiuWu2010Prop3}, we have
	\begin{eqnarray*}
		\mathbb{P}(|\tilde{f}^{(1)}_{ij,n}(\omega)-\mathbb{E} \hat{f}^{(1)}_{ij,n}(\omega)|\geq 2(n/M_{n})^{1/2+\delta}n^{-1/2} )\leq C_3n^{-H}.
	\end{eqnarray*}

We now seek to bound $\mathbb{E}|\hat{f}^{(1)}_{ij,n}(\omega)-\tilde{f}^{(1)}_{ij,n}(\omega)|$.  Note that
	\begin{align*}
		& \mathbb{E}|\hat{f}^{(1)}_{ij,n}(\omega)-\tilde{f}^{(1)}_{ij,n}(\omega)| \\
		\leq &  \mathbb{E}|\hat{f}^{(1)}_{ij,n}(\omega)-\breve{f}^{(1)}_{ij,n}(\omega)|	+ \mathbb{E}|\breve{f}^{(1)}_{ij,n}(\omega)-\bar{f}^{(1)}_{ij,n}(\omega)| + \mathbb{E}|\bar{f}^{(1)}_{ij,n}(\omega)-\tilde{f}^{(1)}_{ij,n}(\omega)| \\
		\leq & 2C_0Mn^2\exp(-\kappa M) + O(n \exp\{-\kappa M\})+ O(n\rho^m) \leq O(Mn\rho^m),
	\end{align*}
	which implies
	\begin{eqnarray}\label{eq-final-1}
		\mathbb{P}(|\tilde{f}^{(1)}_{ij,n}(\omega)-\mathbb{E} (\tilde{f}^{(1)}_{ij,n}(\omega))|\geq 3(n/M_{n})^{1/2+\delta}n^{-1/2} )\leq C_{3}n^{-H}.
	\end{eqnarray}
	Similarly arguments lead to
	\begin{eqnarray}\label{eq-final-2}
		\mathbb{P}(|\tilde{f}^{(2)}_{ij,n}(\omega)-\mathbb{E} (\tilde{f}^{(2)}_{ij,n}(\omega))|\geq 3(n/M_{n})^{1/2+\delta}n^{-1/2} )\leq C_{3}n^{-H},
	\end{eqnarray}
	and
	\begin{eqnarray}\label{eq-final-3}
		\mathbb{P}\left(\left|\frac{1}{n}\sum_{t=1}^{n}X_{t,i}X_{t,j}-\mathbb{E}\left\{ \frac{1}{n}\sum_{t=1}^{n}X_{t,i}X_{t,j}\right\}\right|\geq n^{-1/2+\delta}\right)\leq C_{4}n^{-H}.
	\end{eqnarray}
	Combining \eqref{eq-final-1}-\eqref{eq-final-3}, we obtain
	\begin{eqnarray}\label{eq-final-4}
		\mathbb{P}(|\tilde{f}_{ij,n}(\omega)-\mathbb{E} (\tilde{f}_{ij,n}(\omega))|\geq 7(n/M_{n})^{1/2+\delta}n^{-1/2} )\leq C_{5}n^{-H},
	\end{eqnarray}
	where $C_4, C_5 > 0$ are constants only depending on $\delta, H, \kappa$ and $C_0$.

It follows from a slight modification of Theorem~10.4.1 in \cite{BrockwellDavis2006} and Assumption~\ref{assump-c1} that there exists a constant $C_6 > 0$ only depending on $\kappa$ and $C_0$ such that
	\begin{equation}\label{eq-final-5}
		\max_{i,j}|f_{ij}(\omega)-\mathbb{E} (\tilde{f}_{ij,n}(\omega))| \leq C_6 M_{n}/n.
	\end{equation}
	Combining \eqref{eq-final-4} and \eqref{eq-final-5} yields that for any $(i, j)$ we have
	\begin{eqnarray*}
		\mathbb{P}(|\tilde{f}_{ij,n}(\omega)-f_{ij}(\omega)|\geq C_6M_{n}/n+8(n/M_{n})^{1/2+\delta}n^{-1/2} )\leq C_{5}n^{-H}.
	\end{eqnarray*}

Therefore, using the union bound argument we can show that there exists a constant $C > 0$ depending only on $\kappa$ and $C_0$ such that for any $\delta > 0$ and $H > 0$ the following holds
	\begin{align*}
		\mathbb{P}\left\{\sup_{k \in \{-\lfloor (n-1)/2 \rfloor, \ldots, \lfloor n/2 \rfloor\}} \max_{i, j = 1, \ldots, p} |\widetilde{f}_{ij,n}(\omega_k)-f_{ij}(\omega_k)| > C M_{n}/n+8(n/M_{n})^{1/2+\delta}n^{-1/2} \right\} \leq p^2n^{-H}.
	\end{align*}

\end{proof}

\begin{proof}[Proof of \Cref{thm-main}]
It is due to Theorem~7.2 in \cite{CaiEtal2016} that for any symmetric matrix $A$ and $w \in [1, \infty]$, the relation $\|A\|_w \leq \|A\|_1$ holds; therefore it is enough to consider only the $w = 1$ case. 

Define the event
	\[
		\mathcal{A}_n := \left\{\sup_{\omega}\max_{k, l = 1, \ldots, p} |\widetilde{f}_{ij,n}(\omega)-f_{ij}(\omega)| \leq C\left(\frac{M_{n}}{n}+\frac{1}{\sqrt{M_{n}}}\Big{(}\frac{n}{M_{n}}\Big{)}^{\delta}\right)\right\},
	\]
	where $C > 0$ and $\delta > 0$ are constants.  It follows from Therorem~\ref{thm-key} and Assumption~\ref{assump-c2} that 
	\[
		\mathbb{P}\{\mathcal{A}_n\} \to 1, 
	\]
	as $n \to \infty$.
	
Let $\mathcal{O} := \{-\lfloor (n-1)/2 \rfloor, \ldots, \lfloor n/2 \rfloor\}$.  In the event $\mathcal{A}_n$, we have
	\begin{align*}
		& \sup_{\omega \in \mathcal{O}}\|\widehat{\Sigma}(\omega)\Theta(\omega) - I\|_{\infty}  \leq \sup_{\omega}\|\Theta(\omega)\|_1 \sup_{\omega}\|\widehat{\Sigma}(\omega) - \Sigma(\omega)\|_{\infty} \leq \sup_{\omega}\|\Theta(\omega)\|_1 \sup_{\omega}\|\widetilde{f}_T(\omega) - f(\omega)\|_{\infty} \\
		\leq & \frac{CM_{n, p}M_n}{n} + \frac{CM_{n, p}n^{\delta}}{M_n^{1/2+\delta}} \asymp \lambda,
	\end{align*}
	where $M_{n, p}$ is defined in \eqref{eq-sparsity-definition}.	Then due to the definition of $\widehat{\Theta}(\omega)$, we have for any $\omega$, on the event $\mathcal{A}_n$, $\|\widehat{\Theta}(\omega)\|_1 \leq \|\Theta(\omega)\|_1 \leq M_{n, p}$.
	
Therefore, in the event $\mathcal{A}_n$,
	\begin{align*}
		& \sup_{\omega \in \mathcal{O}}\|\widehat{\Theta}(\omega) - \Theta(\omega)\|_{\infty} = \sup_{\omega \in \mathcal{O}}\|(\Theta(\omega) \widehat{\Sigma}(\omega) - I)\widehat{\Theta}(\omega) + \Theta(\omega)(I - \widehat{\Sigma}(\omega)\widehat{\Theta}(\omega))\|_{\infty}	\\
		\leq & \sup_{\omega \in \mathcal{O}}\|\widehat{\Theta}(\omega)\|_1 \sup_{\omega  \in \mathcal{O}}\|\Theta(\omega) \widehat{\Sigma}(\omega) - I\|_{\infty} + \sup_{\omega \in \mathcal{O}}\|\Theta(\omega)\|_1 \sup_{\omega \in \mathcal{O}}\|I - \widehat{\Sigma}(\omega)\widehat{\Theta}(\omega)\|_{\infty} \\
		\leq & 2M_{n, p}\lambda := t_n.
	\end{align*}

Moreover, we are to bound the $\ell_1$ errors.  it follows from Lemma~7.1 in \cite{CaiEtal2016} that in the event $\mathcal{A}_n$ we have
	\begin{align*}
		\sup_{\omega \in \mathcal{O}}\|\widehat{\Theta}(\omega) - \Theta(\omega)\|_1 \leq 12 c_{n, p}t_n^{1-q},
	\end{align*}
	where $c_{n, p}$ is defined in \eqref{eq-sparsity-definition}, and we complete the proof.

\end{proof}

\end{document}